\documentclass[11pt,reqno,a4paper]{amsart}
\usepackage{a4wide,verbatim}

\title[Equivariant minimal immersions]{On the connected components of the moduli space of equivariant minimal surfaces in $\CH^2$.}

\author{Ian McIntosh}
\address{Department of Mathematics\\ University of York\\
York YO10 5DD, UK}
\email{ian.mcintosh@york.ac.uk}
\date{July 8, 2021}

\newcommand{\Z}{\mathbb{Z}}

\newcommand{\C}{\mathbb{C}}
\newcommand{\Ct}{\mathbb{C}^\times}

\newcommand{\R}{\mathbb{R}}

\renewcommand{\H}{\mathbb{H}}

\newcommand{\CP}{\mathbb{CP}}
\newcommand{\CH}{\mathbb{CH}}
\newcommand{\RH}{\mathbb{RH}}

\renewcommand{\P}{\mathbb{P}}

\newcommand{\caA}{\mathcal{A}}

\newcommand{\caC}{\mathcal{C}}
\newcommand{\caD}{\mathcal{D}}

\newcommand{\caF}{\mathcal{F}}

\newcommand{\caH}{\mathcal{H}}
\newcommand{\caI}{\mathcal{I}}

\newcommand{\caM}{\mathcal{M}}
\newcommand{\caN}{\mathcal{N}}
\newcommand{\caO}{\mathcal{O}}
\newcommand{\caP}{\mathcal{P}}

\newcommand{\caR}{\mathcal{R}}
\newcommand{\caS}{\mathcal{S}}
\newcommand{\caT}{\mathcal{T}}
\newcommand{\caU}{\mathcal{U}}
\newcommand{\caV}{\mathcal{V}}
\newcommand{\caW}{\mathcal{W}}

\newcommand{\caZ}{\mathcal{Z}}

\newcommand{\KER}{\mathit{Ker}}
\newcommand{\COKER}{\mathit{Coker}}

\newcommand{\fE}{\mathfrak{E}}

\newcommand{\so}{\mathfrak{so}}

\newcommand{\su}{\mathfrak{su}}

\newcommand{\fg}{\mathfrak{g}}

\newcommand{\fh}{\mathfrak{h}}
\newcommand{\fm}{\mathfrak{m}}

\newcommand{\fu}{\mathfrak{u}}

\newcommand{\Aut}{\operatorname{Aut}}

\newcommand{\End}{\operatorname{End}}

\newcommand{\Hom}{\operatorname{Hom}}
\newcommand{\Sec}{\operatorname{Sec}}

\newcommand{\im}{\operatorname{im}}

\newcommand{\ad}{\operatorname{ad}}

\newcommand{\tr}{\operatorname{tr}}

\newcommand{\grad}{\operatorname{grad}}

\newcommand{\0}{\mathbf{0}}

\newcommand{\caMtau}{\caM(\Sigma,\CH^2)_\tau}
\newcommand{\oM}{\overline{\mathcal{M}}}
\newcommand{\oV}{\overline{\mathcal{V}}}
\newcommand{\oU}{\overline{\mathcal{U}}}

\newtheorem{thm}{Theorem}[section]
\newtheorem{prop}[thm]{Proposition}
\newtheorem{lem}[thm]{Lemma}
\newtheorem{cor}[thm]{Corollary}

\theoremstyle{remark}

\newtheorem{rem}{Remark}[section]

\numberwithin{equation}{section}

\begin{document}

\begin{abstract}
An equivariant minimal surface in $\CH^n$ is a minimal map of the Poincar\'{e} disc into $\CH^n$ which intertwines 
two actions of the fundamental group of a closed surface $\Sigma$: a Fuchsian representation on the disc and an
irreducible action by isometries on $\CH^n$. The moduli space of these can been studied by relating
it to the nilpotent cone in each moduli space of $PU(n,1)$-Higgs bundles over the conformal surface corresponding to
the map.  By providing a necessary condition for points
on this nilpotent cone to be smooth this article shows that away from the points corresponding to branched minimal immersions 
or $\pm$-holomorphic immersions the moduli space is smooth. The argument is easily adapted to show that for $\RH^n$ the
full space of (unbranched) immersions is smooth. 
For $\CH^2$ we show the connected components of the moduli space of minimal 
immersions are indexed by the Toledo invariant and the Euler number of the normal bundle of the immersion. 
This is achieved by studying the limit points of the $\Ct$-action on the nilpotent cone. 
It is shown that the limit points as $t\to 0$ lead only to branched minimal immersions or $\pm$-holomorphic immersions.
In particular, the Euler number of the normal bundle can only jump by passing through branched minimal maps.
\end{abstract}

\maketitle

\section{Introduction.}

Let $\Sigma$ be a closed oriented surface of genus $g\geq 2$ and let $N$ be a noncompact irreducible symmetric space. Denote
by $G$ the identity component of the isometry group of $N$.
By an equivariant minimal surface we mean a minimal immersion $f:\caD\to N$ of the Poincar\'{e} disc $\caD$
which intertwines the action of a Fuchsian representation $c:\pi_1\Sigma\to\Aut(\caD)$ with an irreducible
representation $\rho:\pi_1\Sigma\to G$. By a theorem of Corlette \cite{Cor} (generalising a theorem of Donaldson 
for $N=\RH^3$ \cite{Don}), the triple $(f,c,\rho)$ is uniquely
determined by the pair $(c,\rho)$, and if one considers that $f$ is essentially unchanged by pre-composition or post-composition
by isometries, then one only needs the conjugacy classes of $c$ and $\rho$. Hence there is a natural way to assign a
topology to the set $\caM(\Sigma,N)$ of equivalence classes $[f,c,\rho]$ of equivariant minimal surfaces by embedding it
in $\caT_g\times \caR(\pi_1\Sigma,G)$ where $\caT_g$ is the Teichm\'{u}ller space of $\Sigma$ and $\caR(\pi_1\Sigma,G)$ is
the character variety of $G$., i.e.,  the moduli space of reductive representations of $\pi_1\Sigma$ into $G$ up to conjugacy.

This topology for $\caM(\Sigma,N)$ was first proposed by Loftin and the author in \cite{LofM18} but prior to this, in \cite{LofM19}, 
we had studied the set $\caM(\Sigma,\CH^2)$ and provided parametrisations for certain subsets which we called
``components'', simply to mean ``parts whose union is the whole''. We showed that each parametrisation equipped 
the component with
the structure of a connected smooth complex manifold, most of which had the same dimension. These parametrisations exploited the nonabelian
Hodge correspondence, through which one can identify an equivariant minimal surface with a $PU(2,1)$-Higgs bundle having
certain properties. However, since in \cite{LofM19} we did not equip the total space $\caM(\Sigma,\CH^2)$ 
with a topology a priori, the question of its connected components had no framework in which to be addressed. One of the aims of this 
article is to answer this question now, given the topology described above (in a nutshell, the components identified in
\cite{LofM19} are, with a few exceptions, the connected components). A key to understanding the connected components 
lies in having a criterion for points of the moduli space to be smooth, and this is provided below for $\caM(\Sigma,N)$ 
when $N$ is either $\CH^n$ or $\RH^n$. The second aim of this article is to show how these components fit together inside
the larger set $\oM(\Sigma,\CH^2)$ obtained by including all branched minimal surfaces. We show that this is not smooth, since
singularities occur where the closures of the connected components of $\caM(\Sigma,\CH^2)$ meet. Note that $\oM(\Sigma,\CH^2)$
was the space studied in \cite{LofM19}, and the distinctive role of branched immersions explained here was not at all apparent
during the writing of \cite{LofM19}.

To summarise the results below, first recall that a $PU(n,1)$-Higgs bundle can be represented by a pair
$(E,\Phi)$ where $E\to\Sigma_c$ is a holomorphic vector bundle of rank $n+1$ over the Riemann surface $\Sigma_c\simeq\caD/c$.
This has a decomposition $E=V\oplus 1$,
where $V$ is rank $n$ and $1$ denotes the trivial line bundle. Then $\Phi$ is a holomorphic one-form with values in
$\Hom(1,V)\oplus\Hom(V,1)$. We write $\Phi=(\Phi_1,\Phi_2)$ to denote the two components. 
In general $(E,\Phi)$ corresponds to an equivariant,
but possibly branched, minimal surface when $\tr(\Phi^2)=0$. Zeroes of $\Phi$ are branch points of the minimal surface 
and therefore we require $\Phi$ to have no zeroes (i.e., $\Phi_1,\Phi_2$ have no common zeroes) to obtain a point in 
$\caM(\Sigma,\CH^n)$.  
From \cite{LofM19} we know that one of $\Phi_1,\Phi_2$ is identically zero precisely when the minimal immersion
is either holomorphic or anti-holomorphic. It is convenient to write $\caM(\Sigma,\CH^n)$ as a disjoint union $\caV\cup\caW$ where
$\caW$ is the subvariety of $\pm$-holomorphic maps and $\caV$ is the complement.
Our first main result concerns $\caV$, and in fact applies also to $\caM(\Sigma,\RH^n)\subset \caV$ using the totally geodesic
embedding of $\RH^n$ into $\CH^n$. 
\begin{thm}\label{thm:smooth}
Both $\caV\subset\caM(\Sigma,\CH^n)$ and $\caM(\Sigma,\RH^n)$ are smooth manifolds. They have complex dimension
$(g-1)\dim(G)$ for $G=PU(n,1)$ and $SO_0(n,1)$ respectively.
\end{thm}
In fact what we prove is that in the nilpotent cone $\caN^c$ (i.e, the locus of $\tr(\Phi^2)=0$ in the Higgs bundle moduli space 
$\caH(\Sigma_c,G)$ for $G=PU(n,1)$ or $SO_o(n,1)$) the regular points of the Hitchin function $\tr(\Phi^2)$ are precisely
the points for which neither $\Phi_j$ has a zero.

Next we restrict our attention to $n=2$.  For $G=PU(2,1)$ the connected components of $\caR(\pi_1\Sigma,G)$ are
indexed by the Toledo invariant, $\tau$. Using Toledo's convention, as in \cite{LofM19}, this satisfies $\tau\in\tfrac23\Z$ and
$|\tau|\leq 2(g-1)$.  We denote the subspace of triples $[f,c,\rho]$ for which $\rho$ has Toledo invariant $\tau$ by $\caMtau$ 
and it is clear that these are disconnected from each other.  Minimal immersions which are not 
$\pm$-holomorphic have complex and anti-complex points (points $p\in\Sigma$ for which $f_*T_p^{1,0}\Sigma$ is a 
subspace of either $T'\CH^2$ or $T''\CH^2$). These give effective divisors, $D_2$ and $D_1$
respectively. From \cite{LofM19} their degrees $d_j=\deg(D_j)$ satisfy
\begin{eqnarray}
\tau = \tfrac23 (d_2-d_1), & \chi(T\Sigma^\perp) = 2(g-1)-d_1-d_2, \label{eq:tau} \\
0\leq 2d_1+d_2< 6(g-1),&0\leq d_1+2d_2< 6(g-1),\label{eq:d1d2}
\end{eqnarray}
where $T\Sigma^\perp$ is the normal bundle of $f$, which means the quotient of $T\caD^\perp\subset
f^{-1}T\CH^2$ by the natural action of $\pi_1\Sigma$. The inequalities are necessary and sufficient conditions for the
existence of minimal immersions which are not $\pm$-holomorphic. In particular, the degrees $d_1,d_2$ are equivalent information to the
pair $\tau$ and $\chi(T\Sigma^\perp)$. For holomorphic immersions the only topological invariant is the Toledo invariant, which
determines $\chi(T\Sigma^\perp)$ through
\begin{equation}\label{eq:hol}
\chi(T\Sigma^\perp) = -\tfrac32\tau +2(g-1).
\end{equation}
This follows directly by taking the quotient of the direct sum of holomorphic bundles $f^{-1}T'\CH^2 = T^{1,0}\caD \oplus T\caD^\perp$ 
by the $\pi_1\Sigma$-action.

Let $\caV(d_1,d_2)\subset\caV$ denote the subset of those immersions whose
divisors have these degrees fixed. It was shown in \cite{LofM19} that each $\caV(d_1,d_2)$ can be parametrised by 
a smooth connected complex manifold of dimension $8(g-1)$ 
(actually in \cite{LofM19} we included the branched minimal immersions, which occur when $D_1\cap D_2\neq\emptyset$, 
but as these are described by a complex analytic subvariety 
removing these does not affect the connectedness of $\caV(d_1,d_2)$). 
In Lemma \ref{lem:Z} below we show that the parametrisation in \cite{LofM19} is smoothly compatible with the structure 
it inherits from Theorem \ref{thm:smooth}. Since $\caV(d_1,d_2)$
all have the same dimension they are therefore the connected components  of $\caV$. To establish the connected components
of $\caM(\Sigma,\CH^2)$ it remains to see how these fit together with points of $\caW$. Set $\caW_\tau=
\caW\cap\caM(\Sigma,\CH^2)_\tau$ for each $\tau$. We need only consider $\tau\geq 0$ since
$\caM(\Sigma,\CH^2)$ carries a natural real involution $[f,c,\rho]\mapsto [\bar f,c,\bar \rho]$ for which
$\tau(\bar\rho)=-\tau(\rho)$.
\begin{thm}\label{thm:connected}
Let $\caMtau$ denote the space of equivariant minimal immersions with Toledo invariant $\tau$. Then $\caM(\Sigma,\CH^2)_0$ has
connected components $\caV(d_1,d_2)$ for $d_2=d_1$. For $0<\tau<2(g-1)$ the connected components of $\caMtau$ are
$\caV(d_1,d_2)$ with $d_1\neq 0$ and $\caV(0,d_2)\cup\caW_\tau$, while $\caMtau=\caW_\tau$ for $\tau=2(g-1)$. 
\end{thm}
In particular, the connected components are indexed by $\tau$ and $\chi(T\Sigma^\perp)$: when $d_1=0$ the two equations 
\eqref{eq:tau} become the single equation \eqref{eq:hol}. 

This theorem is proved through a complete description of the limit points of the $\Ct$-action $t\cdot(E,\Phi)=(E,t\Phi)$ 
on $PU(2,1)$-Higgs bundles, for both $t\to\infty$ and $t\to 0$. As well as proving
the theorem, it gives us information about how the closures of the connected components fit together inside the space 
$\oM(\Sigma,\CH^2)$ which includes all branched minimal immersions. On the Higgs bundle side this corresponds to the space
of all $(c,E,\Phi)$ for which $(E,\Phi)$ lies in the nilpotent cone $\caN^c$
excluding the locus $\Phi=0$ (since this corresponds to constant maps).
Each $\caN^c$ is known to be stratified by the unstable manifolds of the downwards Morse flow for the Higgs
field energy $\|\Phi\|^2_{L^2}$: for each connected component $C$ of critical points these agree with
\[
\caU_C = \{(E,\Phi):\lim_{t\to\infty}(E,t\Phi)\in C\}.
\]
It was conjectured in \cite{LofM19}, and is proven below, that except when $C$ consists of local minima (there is
precisely one such $C$ for each $\tau$ and it equals $\caW_\tau$) these unstable manifolds
correspond to the fibres of the analytic family $\caV(d_1,d_2)$ over $\caT_g$. We understand how these fit
together by examining the limit in the opposite direction, as
$t\to 0$. It is well-known that when $E$ is semi-stable (as a vector bundle) the limit $\lim_{t\to 0}(E,t\Phi)$ must have
$\Phi=0$ and therefore this can only occur when $\tau=0$. When $E$ is strictly unstable (and assuming $\tau\geq 0$ without loss of
generality) we show that either $(E,\Phi_1)$ is a Hodge bundle, in which case this is limit, or the limit is determined by
the maximal destabilizing line subbundle of $V$ and always corresponds to a branched minimal immersion. 

Thus we have the the following geometric picture of how these pieces fit together. Each fibre over $c$ of the space 
$\oM(\Sigma,\CH^2)_\tau$ is a union of these unstable
manifolds, and their closures intersect either at branched minimal surfaces or at $\pm$-holomorphic surfaces. Removing
these leaves a disjoint union of components $\caV(d_1,d_2)$, one for each critical manifold of non-minima. In particular, one cannot
pass from one component $\caV(d_1,d_2)$ to another in $\oM(\Sigma,\CH^2)$ without forcing the minimal immersion to branch somewhere. 
Another way of saying this is that $\chi(T\Sigma^\perp)$ can jump but only by passing through branched maps. 

%Let me conclude with some comments about the deeper questions driving this investigation. 

\textbf{Acknowledgements.} The initial stages of this research were carried out while the author was a visitor at the
Simons Center for Geometry and Physics, Stony Brook University, in April 2019 during the programme ``The Geometry and Physics of Hitchin
Systems'', and then subsequently at the MFO Oberwolfach workshop ``The Geometry and Physics of Hitchin Bundles'' in May
2019. I gratefully acknowledge the support of these institutions, and thank the organisers for the opportunity to
participate. I also thank Andy Sanders and John Loftin for the stimulating discussions during the SCGP programme which raised
some of the questions answered here.

\section{Definitions.}
Let $\C^{n,1}$ denote $\C^{n+1}$ equipped with the pseudo-Hermitian inner product characterised by
\[
\langle u,u\rangle = |u_1|^2+\ldots |u_n|^2-|u_{n+1}|^2.
\]
Let $e_1,\ldots,e_{n+1}$ denote the standard basis vectors for $\C^{n+1}$. 
We model $\CH^n$ on the space of lines in $\CP^n$ generated by a negative definite vector and fix a base
point $o=[e_{n+1}]$. 
Throughout this article we set $G=PU(n,1)$, the group of orientation preserving isometries of $\CH^n$, and let $H\simeq U(n)$ 
denote the maximal compact subgroup which fixes the base point $o$. The Lie algebra $\fg=\su(n,1)$ of $G$ has Cartan decomposition 
$\fg=\fh\oplus \fm$ corresponding to the decomposition of elements 
\begin{equation}\label{eq:h+m}
\begin{pmatrix} A & u \\ u^\dagger & a\end{pmatrix} = 
\begin{pmatrix} A & 0 \\ 0 & a\end{pmatrix} + 
\begin{pmatrix} 0 & u \\ u^\dagger & 0\end{pmatrix} 
\end{equation}
where $A\in \fu(n)$, $a=-\tr(A)$, $u\in\C^n$ and $u^\dagger$ is the Hermitian transpose.

Let $\Sigma$ be a closed oriented surface of genus $g\geq 2$ and let $\caD$ denote the Poincar\'{e} disc thought of
as the universal cover of $\Sigma$ (note that $\caD\simeq\CH^1$). A Fuchsian representation $c:\pi_1\Sigma\to PU(1,1)$ equips $\Sigma$ with
a complex structure which will be denoted $\Sigma_c$. We will assume this complex structure is compatible with the
orientation of $\Sigma$ and hence the set conjugacy classes of such Fuchsian representations can be identified with the Teichm\"{u}ller space
$\caT_g$ of $\Sigma$. An equivariant minimal surface in $\CH^n$ is the equivalence class
$[f,c,\rho]$ of a triple consisting of a Fuchsian representation $c$, an indecomposable representation
$\rho:\pi_1\Sigma\to G$ and a minimal immersion $f:\caD\to\CH^n$ which intertwines the actions of $c$ and $\rho$, i.e., 
\[
f\circ c(\delta) = \rho(\delta)\circ f,\quad \forall\delta\in\pi_1\caD.
\]
Equivalence is with respect to the action of $PU(1,1)\times G$ by conjugation on $(c,\rho)$ and the corresponding natural
action on $f$. Indecomposability of $\rho$ means exactly that its image does not lie in a proper subgroup of $G$. By the
same argument as \cite[Lemma 2.3]{LofM18} it is equivalent to the condition that $f$ is linearly full (i.e., its image
does not lie in a totally geodesic copy of $\CH^k$ for $k<n$).
\begin{rem}
We could also say $\rho$ is irreducible in its standard real representation on $\C^{n,1}$, although there is the potential 
for confusion with other meanings of irreducibility (say, for the adjoint representation of the complexification $G^\C$). 
\end{rem}
We use $\caM(\Sigma,\CH^n)$ to denote the set of equivariant minimal surfaces and give it the topology
of its natural embedding
\begin{equation}\label{eq:caF}
\caF:\caM(\Sigma,\CH^n)\to \caT_g\times \caR(\pi_1\Sigma,G);\quad [f,c,\rho]\to [c,\rho],
\end{equation}
into the product of Teichm\"{u}ller space with the character variety $\caR(\pi_1\Sigma,G)$ of reductive
representations of $\pi_1\Sigma$ into $G$ modulo conjugacy. This map is injective by Corlette's uniqueness theorem
for equivariant harmonic maps under these conditions \cite{Cor}. We will denote by $\caM(\Sigma_c,\CH^n)$ the fibre
over fixed $c\in\caT_g$.

Recall that, for each choice of $c$, non-abelian Hodge theory provides a homeomorphism between $\caR(\pi_1\Sigma,G)$ and the 
moduli space $\caH(\Sigma_c,G)$ of polystable $G$-Higgs bundles over $\Sigma_c$ (and this homeomorphism is analytic
away from singularities). Indeed,  for any simple noncompact $G$ there is a homeomorphism between $\caT_g\times\caR(\pi_1\Sigma,G)$ 
and the universal Higgs bundle moduli space $\caH(\caC_g,G)$ over $\caT_g$
\cite[Thm 7.5]{AleC} (here $\caC_g$ denotes the universal Teichm\"{u}ller curve over $\caT_g$). 
This is a complex analytic space for which
the fibre over $c\in \caT_g$ is biholomorphic to $\caH(\Sigma_c,G)$. This homeomorphism is smooth about smooth
points and therefore we obtain an embedding
\begin{equation}\label{eq:caF'}
\caF':\caM(\Sigma,\CH^n)\to \caH(\caC_g,G),
\end{equation}
when $G=PU(n,1)$. This equips $\caM(\Sigma,\CH^n)$ with a complex analytic structure. 

We can easily characterise the image of $\caF'$.
Recall that in the projective equivalence class of the $G$-Higgs bundle $(E,\Phi)$, 
$E$ can be taken to have the form $V\oplus 1$ where $V$ is
a holomorphic rank $n$ bundle over $\Sigma_c$ and  $1$ denotes the trivial line bundle. The Higgs field $\Phi$ satisfies
\[
\Phi\in (\Hom(1,V)\oplus\Hom(V,1))\otimes K,
\]
and we will write $\Phi=(\Phi_1,\Phi_2)$ to indicate the two components for this decomposition. It is well-known
that the harmonic map corresponding to $(E,\Phi)$ is weakly conformal (therefore branched minimal) when $\tr(\Phi^2)=0$.
The branch points correspond exactly to the zeroes of $\Phi$. Consequently we have the following elementary observation.
\begin{lem}
The pair $(E,\Phi)$ lies in the image of $\caF'$ if and only if it is stable (hence indecomposable) with $\tr(\Phi^2)=0$ and
$\Phi$ having no zeroes.
\end{lem}
Each equivariant minimal surface $(f,c,\rho)$ has associated to it
the Toledo invariant $\tau$ of $\rho$ and we denote the subset of triples with fixed $\tau$ by $\caM(\Sigma,\CH^n)_\tau$.
Our convention will be that $\tau=-\frac{2}{n+1}\deg(V)$ so that $|\tau|\leq 2(g-1)$: this matches Toledo's original
definition and is the one used in \cite{LofM19}. 

We will also write $\caM(\Sigma,\CH^n)$ as a disjoint union $\caV\cup\caW$ where $\caW$ consists of all $\pm$-holomorphic
immersions and $\caV$ is the complement. Each of these has its subsets $\caV_\tau$, $\caW_\tau$ containing those elements
with fixed Toledo invariant. We note that under the non-abelian Hodge correspondence elements of $\caW$ correspond to
Higgs bundles with $\Phi_2=0$ for $\tau>0$ or $\Phi_1=0$ for $\tau<0$. Since we always have $\Phi\neq 0$ we know
$\caW_0=\emptyset$. Note that $\caW_\tau$ corresponds to the length two Hodge bundles \cite{LofM19}, which are the minima
of the Higgs field energy $\|\Phi\|^2_{L^2}$ \cite{BraGPG}.

From \cite{LofM19} we know that when $n=2$ the space $\caV$ is a disjoint union of subsets $\caV(d_1,d_2)$, corresponding 
to $PU(2,1)$-Higgs bundles for which $V$ and $\Phi$ are determined by an exact sequence \cite{LofM19} 
\begin{equation}\label{eq:exactV}
0\to K^{-1}(D_1) \stackrel{\Phi_1}{\to} V \stackrel{\Phi_2}{\to} K(-D_2)\to 0,
\end{equation}
where $d_j=\deg(D_j)$ satisfy the inequalities \eqref{eq:d1d2}. There is a one-to-one correspondence between Higgs bundles of this
type and data $(c,D_1,D_2,\xi)$ where $\xi$ is the extension class of this extension: this was used in \cite{LofM19} to
give each $\caV(d_1,d_2)$ the structure of a complex manifold of dimension $5(g-1)$. 

Finally, we will want to consider the strictly larger space $\oM(\Sigma,\CH^2)$ which includes additionally the 
branched minimal immersions. Under $\caF'$ it maps to the locus $\tr(\Phi^2)=0$ excluding the Higgs bundles with $\Phi=0$.
We will write $\bar\caV$ and $\bar\caW$ when we extend $\caV$ and $\caW$ to 
include branched minimal immersions of the respective types. The total space is disconnected into Toledo invariant 
pieces $\oM(\Sigma,\CH^2)_\tau$ and for $\tau\neq 0$ each of these is connected. This follows from the fact that for
$\tau\neq 0$ each fibre over $c\in\caT_g$ agrees with the component $\caN^c_\tau$ of the nilpotent cone in Toledo invariant 
$\tau$, and these are connected by the  downwards Morse flow. For $\tau=0$ this fibre is $\caN^c_0$ without the locus $\Phi=0$. Since
this locus is the critical manifold of minima we can no longer use the Morse flow argument, and an understanding of
the connected components would require a different approach.

\section{Smooth points of $\caM(\Sigma,N)$.}

For $N=\CH^n$ or $\RH^n$ let $\caV\subset \caM(\Sigma,N)$ denote the subvariety of points where neither $\Phi_j$ is
identically zero. For $\RH^n$ this is the whole space while for $\CH^n$ it consists of 
the equivariant minimal surfaces which are not $\pm$-holomorphic. 
Because we assume $\rho$ is indecomposable the there are no singular points of $\caT_g\times \caR(\pi_1,G)$ in $\caV$.

Recall that $\caC_g$ denotes the universal Teichm\"{u}ller curve over $\caT_g$.  If we denote by $H^0(\caC_g,K^2)$ the 
bundle over $\caT_g$ with fibre $H^0(\Sigma_c,K^2)$ then the map
\[
\caI:\caH(\caC_g,G)\to H^0(\caC_g,K^2);\quad (c,E,\Phi)\mapsto (c,\tr(\Phi^2)),
\]
is complex analytic. 
We will denote  its zero locus (i.e., the preimage of the zero section of $H^0(\caC_g,K^2)$) by $\caN$: for real rank one groups 
like $PU(n,1)$ and $SO_0(n,1)$ this is the union $\cup_{c\in\caT_g}\caN^c$ of all nilpotent cones. 

Let $\caN^o\subset\caN$ denote the open subvariety for which $\Phi$ vanishes nowhere and neither $\Phi_j$ is identically zero. 
Then $\caN^o$ is the image of $\caV$ in $\caH(\caC_g,G)$ and Theorem \ref{thm:smooth} is the statement that $\caN^o$ is smooth. 
We will prove this by showing that points of $\caN^o$ are regular points of $\caI$.  Since $\caI$ 
maps $\caH(\Sigma_c,G)$ to $H^0(\Sigma_c,K^2)$ it suffices the show that each point of $(\caN^c)^o=\caN^o\cap\caN^c$ 
is a regular point of $\caI_c$, the restriction of $\caI$ to the fibre over $c$.
We do this by generalising an argument used by Hitchin for $SL(2,\C)$-Higgs bundles \cite{Hit19}
to compute the rank of $d\caI_c$. 
For this we must recall the hypercohomology description of the tangent space to $\caH(\Sigma_c,G)$ \cite{BisR}. 

Let $P\to\Sigma_c$ denote the 
principal $H^\C$-bundle whose associated bundle is $E$, and let $P(\fh^\C)$, $\caP(\fm^\C)$ denote the holomorphic 
$\fh^\C$ and $\fm^C$ bundles associated by the adjoint action of $H^\C$ on these subspaces of $\fg^\C$. 
One knows from \cite{BisR} that when $(E,\Phi)$
is stable the tangent space at this point is isomorphic to the first hypercohomology $\H^1(\caA^*)$ of the complex
\begin{equation}
\caA^0\stackrel{\ad\Phi}{\to}\caA^1\to 0\to\ldots
\end{equation}
where $\caA^0$ is the sheaf of local sections of $P(\fh^\C)$ and $\caA^1$ is the sheaf of local sections of
$P(\fm^\C)\otimes K$. As Hitchin \cite{Hit19} points out, by considering the ``second'' spectral sequence for the hypercohomology of
$\caA^*$ (i.e., the spectral sequence for the filtration by the degree of Cech cochains) one obtains an
isomorphism for $\H^1(\caA^*)$ involving the kernel and cokernel sheaves
\begin{equation}\label{eq:KerCoker}
0\to \KER \to \caA^0\stackrel{\ad\Phi}{\to}\caA^1\to\COKER\to 0,
\end{equation}
namely,
\begin{equation}\label{eq:H1}
\H^1(\caA^*) \simeq H^1(\KER)\oplus H^0(\COKER).
\end{equation}
Now let $B$ denote the Killing form on $\fg^\C$, specifically $B(\eta,\xi) = \tr(\ad\eta\ad\xi)$. 
By adjoint invariance
$B(\Phi,\ad\Phi(\eta))=0$ for all $\eta$, so that we have a well defined map of sheaves
\begin{equation}\label{eq:B}
B(\Phi,\cdot):\COKER\to K^2;\quad \eta\mapsto B(\Phi,\eta),
\end{equation}
and an induced map from $H^0(\COKER)$ to $H^0(K^2)$.
When pre-composed with the projection of $\H^1(\caA^*)$ onto $H^0(\COKER)$ from \eqref{eq:H1}, at a stable
point $(E,\Phi)$, this induced map agrees with $d\caI_c$ \cite[Remark 2.8(iv)]{BisR}. 
Hence the rank of $d\caI_c$ at $(E,\Phi)$ equals the rank of $B(\Phi,\cdot)$ on $H^0(\COKER)$. 
\begin{lem}\label{lem:rank}
Suppose $\ad\Phi$ has corank $1$, i.e., $\COKER$ is a rank $1$ sheaf. Then $d\caI_c$ has maximal rank if and only if 
$\Phi$ has no zeroes. 
\end{lem}
\begin{proof}
When $\COKER$ is rank one $B(\Phi,\cdot)$ vanishes precisely at the zeroes of $\Phi$ with the same divisor of zeroes, $D$, 
and hence $\COKER\simeq K^2(-D)$. Thus the
induced map $H^0(\COKER)\to H^0(K^2)$ is maximal rank, and hence $d\caI$ has maximal rank, if and only if $\Phi$ has no zeroes. 
\end{proof}
Now our aim is to show that $\ad\Phi$ has corank $1$ for points of $\caV$. First we do the case $N=\CH^n$.

%-----------------
\begin{comment}
\begin{prop}\label{prop:dI}
The differential $d\caI_c$ has maximal rank precisely at the pairs $(E,\Phi)$ for which $\Phi$ has no zeroes and neither 
$\Phi_j$ is identically zero.
\end{prop}
Proposition \ref{prop:dI} follows from the next lemma.
\end{comment}
%------------------
\begin{lem}
For $G=PU(n,1)$, when neither $\Phi_j$ is identically zero the map $\ad\Phi$ has corank one. 
\end{lem}
\begin{proof}
Following \cite{LofM19}, since $\tr(\Phi^2)=0$ we have $\Phi_2\circ\Phi_1=0$, and so we have a sequence
\begin{equation}\label{eq:extV}
0\to K^{-1}(D_1)\stackrel{\Phi_1}{\to} V\stackrel{\Phi_2}{\to} K(-D_2)\to 0,
\end{equation}
where $D_j$ is the divisor of zeroes of $\Phi_j$. About each point $p\in\Sigma$ we can choose a local holomorphic chart
$(U,z)$ and a local holomorphic frame $\sigma_1,\ldots,\sigma_n$ for $V$ over $U$ satisfying
\[
\Phi(\sigma_0) = z^k\sigma_1dz, \quad \Phi(\sigma_n) = z^ldz,\quad \Phi(\sigma_j)=0\ \text{for}\ j\neq n.
\]
Here $k,l$ are, respectively, the degrees of $D_1$ and $D_2$ at the point $p$ and $\sigma_0$ denotes a local trivialising
section of the trivial bundle. With respect to this frame $\Phi$ is represented as a local section of $\End(E)\otimes
K$ by
\begin{equation}\label{eq:xi}
\xi=\begin{pmatrix} 0&\ldots &0& z^k\\ 0&\ddots &&0\\ \vdots & && \vdots \\ 0&\ldots &z^l&0\end{pmatrix}dz=
\begin{pmatrix} 0&z^ke_1\\ z^le_n^t&0\end{pmatrix}dz,
\end{equation}
where in the second expression the standard basis vectors $e_1,\ldots,e_n$ for $\C^n$ have been used to write the matrix
in block form.
Let $\chi$ be a local holomorphic section of $P(\fh^\C)$, identified with a locally holomorphic $\fh^\C$-valued function. 
We compute
\begin{equation}\label{eq:bracket}
[\xi,\chi] = \left[ \begin{pmatrix} 0&z^ke_1\\ z^le_n^t&0\end{pmatrix}dz, \begin{pmatrix} A&0\\0&a\end{pmatrix}\right]
=\begin{pmatrix} 0 & z^k(aI_n-A)e_1\\ z^le_n^t(A-aI_n) &0\end{pmatrix}dz,
\end{equation}
$a=-\tr(A)$. Hence $[\xi,\chi]=0$ holds over $U$ when the $2n$ equations 
\[
(A+\tr(A))e_1=0,\quad (A^t+\tr(A))e_n=0, 
\]
hold. In components these equations are
\[
2A_{11}+\sum_{i=2}^nA_{ii}=0,\ A_{j1}=0,\ A_{nk}=0,\ 2A_{nn}+\sum_{i=1}^{n-1}A_{ii}=0,
\]
for $2\leq j\leq n$ and $1\leq k\leq n-1$. We see that this gives $2n-1$ independent equations since $A_{n1}=0$ appears
twice. Hence $\ad\xi$ has rank $2n-1=\dim\fm^\C-1$.
\end{proof}
Now we deal with the case $N=\RH^n$ and follow the description of the Higgs bundles for this case given in \cite{LofM18}. 
These have the form 
$E=V\oplus 1$ where $V$ has $\det(V)=1$ and carries a non-degenerate quadratic form $Q_V$. In this case $\Phi_2=\Phi_1^\top$ 
where the transpose is with respect to $Q_V$. It follows that for $f$ to be an immersion both $\Phi_j$ vanish nowhere so both divisors 
$D_1,D_2$ are trivial.
As a smooth bundle $V$ splits into a direct sum $K^{-1}\oplus W\oplus K$ and with respect to this splitting we can write
\begin{equation}\label{eq:Q_V}
Q_V = \begin{pmatrix} 0 & 0 & 1\\ 0& Q_W&0\\ 1 &0 & 0\end{pmatrix}.
\end{equation}
Then $\Phi$ is represented by $\xi$ in \eqref{eq:xi} with $k=l=0$. For the Lie algebra we use
\[
\so(Q,1) = \left\{ \begin{pmatrix} A& u\\ v^t & 0\end{pmatrix}: A^tQ+QA=0,\ Qu=v\right\}\simeq \so(n+1,\C),
\]
where $A\in\End(\C^n)$, $u,v\in\C^n$, $Q$ is \eqref{eq:Q_V} with $Q_W = I_{n-2}$, and the transpose is the usual matrix
transpose. The appropriate (complexified) symmetric space decomposition is obtained from the 
Cartan decomposition \eqref{eq:h+m} above:
\[
\so(Q,1) = \fh_Q^\C\oplus \fm_Q^\C,\quad \fh_Q^\C=\so(Q,1)\cap\fh^\C,\ \fm_Q^\C = \so(Q,1)\cap\fm^\C.
\]
\begin{lem}
For $G=SO_0(n,1)$ the map $\ad\Phi$ has corank one. 
\end{lem}
\begin{proof}
Following the same arguments as for $\CH^n$ above but applied to $\so(Q,1)$ it suffices to show that the map
$\ad\xi:\fh_Q^\C\to\fm_Q^\C$ has corank $1$. Using the equations \eqref{eq:bracket} but with $\chi\in\fm_Q^\C$ we 
we see that the image of $\ad\xi$ is isomorphic to
\[
\{Ae_1:A^t=-QAQ\}\subset \C^n.
\]
The symmetry $A^t=-AQA$ implies that $A_{n1}=0$ and this is the only condition $Ae_1$ must satisfy. Therefore $\im(\ad\xi)$
has codimension one in $\fm_Q^\C$.
\end{proof}
To complete the proof of Theorem \ref{thm:smooth} we note that at regular points the dimension of $\caI^{-1}(0)$ is
\begin{eqnarray*}
\dim\caH(\caC_g,G)-\dim H^0(\Sigma_c,K^2) & =  &\dim(\caT_g) + \dim(\caR(\pi_1\Sigma,G))-\dim H^0(\Sigma_c,K^2)\\
&=& \dim(\caR(\pi_1\Sigma,G).
\end{eqnarray*}
Here $\dim(\caR(\pi_1\Sigma,G))$ means the dimension of the smooth open subvariety of indecomposable representations. 
This has real dimension $(2g-2)\dim(G)$.
\begin{rem}
Theorem \ref{thm:smooth} fills a small gap in the proof of Theorem 4.4 of \cite{LofM18}, which gives a smooth parametrisation of
the connected components of $\caM(\Sigma,\RH^4)$. The proof that the parametrisation is a local diffeomorphism needs the fact 
the $\caM(\Sigma,\RH^4)$ is non-singular.
\end{rem}

\section{Limit points of the $\Ct$-action for $n=2$.}

It is well-known that $\Ct$ acts on $G$-Higgs bundles by $t\cdot(E,\Phi)=(E,t\Phi)$, $t\in\Ct$, and that the fixed points
of this action are Hodge bundles, i.e., one can write $E=\oplus_{j=1}^m E_j$ so that $\Phi:E_j\to E_{j+1}\otimes K$ with
$E_{m+1}=0$. The integer $m$ is called the \emph{length} of the Hodge bundle. It is also well-known that these fixed points 
are exactly the critical points of the Higgs field energy $\fE(E,\Phi)=\|\Phi\|^2_{L^2}$. Given one
of the connected components $C$ of $\Ct$-fixed points define
\[
\caS_C =\{(E,\Phi): \lim_{t\to 0} (E,t\Phi)\in C\}, \quad
\caU_C =\{(E,\Phi): \lim_{t\to \infty} (E,t\Phi)\in C\}.
\]
By a theorem of Kirwan these agree with, respectively, the stable and
unstable manifolds of $C$ along the downwards Morse flow (gradient flow of $-\grad\fE$). 
One knows that $\cup_C \caS_C$ equals $\caH(\Sigma_c,G)$ while $\cup_C\caU_C$ equals the nilpotent cone $\caN^c$.

Our aim in this section is to identify the limit points as $t\to\infty$ and $t\to 0$ inside $\oM(\Sigma,\CH^2)$. 

First we note that if $\Phi_j=0$ for either $j=1$ or $j=2$, the pair $(E,\Phi)$ is plainly a Hodge bundle of length $2$, hence a $\Ct$-fixed
point.  So our interest is when neither $\Phi_j$ vanishes identically.
In this case, using \eqref{eq:exactV}, as a $C^\infty$-bundle
\begin{equation}\label{eq:Vinfty}
V\simeq_{C^\infty} K^{-1}(D_1)\oplus K(-D_2).
\end{equation}
The $\bar\partial$-operator for the holomorphic structure of $V$ can be written with respect to this smooth splitting as
\begin{equation}\label{eq:dbarV}
\bar\partial_V = \begin{pmatrix} \bar\partial_1 & \beta\\ 0&\bar\partial_2\end{pmatrix},
\end{equation}
where $\beta$ is a smooth $(0,1)$-form taking values in $\Hom(K(-D_2),K^{-1}(D_1))$): its Dolbeault cohomology
class $[\beta]$ corresponds to the extension class in $H^1(K^{-2}(D_1+D_2))$ which determines the extension \eqref{eq:exactV}.
When $[\beta]=0$ it is easy to check that this gives a Hodge bundle \cite{LofM19}, which 
we will denote by $(E^\infty,\Phi^\infty)$. Here $E^\infty = V^\infty\oplus 1$ where $V^\infty$ is the holomorphically 
split bundle \eqref{eq:Vinfty} and 
\[
\Phi^\infty_1:K^{-1}(D_1)\to V^\infty,\quad \Phi^\infty_2: V^\infty\to K(-D_2),
\]
are the natural inclusion and projection respectively. Notice that $\Phi^\infty=\Phi$ as a map on smooth bundles.
\begin{lem}
For $(E,\Phi)\in\caV$, $\lim_{t\to \infty} (E,t\Phi) = (E^\infty,\Phi^\infty)$.
\end{lem}
The proof is essentially the same as the proof of Prop.\ 4.9 in \cite{LofM18}. 
\begin{proof}
We perform a gauge transformation. Set
\begin{equation}
g_t=\begin{pmatrix} t&0&0\\ 0&t^{-1}&0\\ 0&0&1\end{pmatrix}.
\end{equation}
Then $g_t\Phi g_t^{-1}=t\Phi$ while
\begin{equation}\label{eq:t}
g_t^{-1}\bar\partial_Eg_t = \begin{pmatrix} 
\bar\partial_1 &  t^{-2}\beta &0\\ 0&\bar\partial_2&0\\ 0&0&0&\bar\partial\end{pmatrix}.
\end{equation}
Let $E^t$ denote the holomorphic bundle for this $\bar\partial$-operator. Then $(E,t\Phi)\simeq (E^t,\Phi)$ hence 
\[
\lim_{t\to\infty}(E,t\Phi) = \lim_{t\to\infty} (E^t,\Phi) = (E^\infty,\Phi^\infty).
\]
\end{proof}
Recall that in \cite{LofM19} the data
$(c,D_1,D_2,[\beta])$ is used to parametrise the component $\caV(d_1,d_2)$ by fixing the
degrees of the divisors. Let $\oV(d_1,d_2)$ denote the closure of this in $\oV$, i.e., where $D_1,D_2$ can have points in
common. By \cite{LofM19} this is a complex analytic family over Teichm\"uller space: let $\oV_c(d_1,d_2)$ denote the fibre
over $c\in\caT_g$. The previous lemma shows that, under the map from $\oV_c(d_1,d_2)$ to Higgs bundles over $\Sigma_c$,  
$(c,D_1,D_2,[t^{-2}\beta])$ maps to $(E,t\Phi)$ and therefore the limit as $t\to\infty$ is given by the
trivial extension for $V$. Given that this parametrisation of Higgs bundles is smooth (which is proved in Lemma
\ref{lem:Z} below) we have the following corollary, which was conjectured in \cite{LofM19}.
\begin{cor}\label{cor:U}
$\oV_c(d_1,d_2)$ is diffeomorphic to the unstable manifold $\,\caU_C$ where $C$ is the critical manifold of Hodge bundles for which
$V=K^{-1}(D_1)\oplus K(-D_2)$ (i.e., the trivial extension).
\end{cor}
Now we turn to limits as $t\to 0$. For a stable Higgs bundle $(E,\Phi)$ with Toledo invariant $\tau$ with $E=V\oplus 1$ 
we note that $\tau = -\tfrac23\deg(V)$ in our convention. Since
the dual Higgs bundle has opposite Toledo invariant we may assume that $\tau\geq 0$. 
\begin{prop}\label{prop:limit0}
Let $(E,\Phi)$ be a stable Higgs bundle with $\tau\geq 0$ and let $(E^0,\Phi^0)=\lim_{t\to 0}(E,t\Phi)$. Suppose
$(E,\Phi)$ is not a Hodge bundle and therefore is given by a non-trivial extension of the form \eqref{eq:exactV} for effective divisors
$D_1,D_2$. Then exactly one of the following holds:
\begin{enumerate}
\item $E$ is a semistable bundle, in which case $\tau=0$ and $(E^0,\Phi^0)=(Gr(E),0)$, where $Gr(E)$ is
the associated graded bundle arising from a Jordan-H\"older filtration of $E$. 
\item $\tau>0$ and every line subbundle $L$ of $V$ satisfies $\deg(L)<\tfrac13\deg(V)$, in which case $(E^0,\Phi^0)=(E,\Phi_1)$
is a length two Hodge bundle. 
\item $\tau>0$, and the maximal destabilizing line subbundle $L$ of $V$ has $\deg(L)=\tfrac13\deg(V)$. 
In this case there is a positive
divisor $D$ for which $L\simeq K(-D_2-D)$, $V/L\simeq K^{-1}(D_1+D)$ and $(E^0,\Phi^0)$ is the 
polystable Higgs bundle 
\begin{equation}\label{eq:caseiii}
(K^{-1}(D_1+D)\oplus 1,\varphi)\oplus (K(-D_2-D),0),
\end{equation}
where $\varphi$ is the projection of $\Phi_1:1\to V\otimes K$ to the quotient. 
\item $\tau\geq 0$, $V$ is unstable and its maximal destabilizing line subbundle $L$ has $\deg(L)>\tfrac13\deg(V)$. In this case
there is a positive divisor $D$ so that $L,V/L$ have the isomorphisms in (iii) but now 
$(E^0,\Phi^0)$ is a length three Hodge bundle of the form
\begin{equation}\label{eq:caseiv}
(K^{-1}(D_1+D)\oplus K(-D_2-D)\oplus 1,\Phi^0),
\end{equation}
where $\Phi^0_1$ is the projection of $\Phi_1$ to $(V/L)\otimes K$, and $\Phi^0_2$ is the restriction of $\Phi_2$ to $L$. 
\end{enumerate}
In case (i) the limit corresponds to a constant harmonic map. In all other cases except (ii) with $d_1=0$ (i.e., $\Phi_1$
has no zeroes) the limit corresponds to a branched minimal surface. In case (iii) these are branched
holomorphic maps to a totally geodesic $\CH^1$ in $\CH^2$.
\end{prop}
\begin{proof}
(i) If $E=V\oplus 1$ is a semistable bundle then $0=\deg(1)\leq\tfrac13\deg(V)$. By assumption $\tau\geq 0$, hence $\tau=0$. 
Now $(E,0)$ is $S$-equivalent to $(Gr(E),0)$ \cite{Ses67}, which is polystable and hence
the limit in moduli space. In particular, note that either $V$ is polystable and $E=Gr(E)$, or $V$ is strictly semistable
and a non-trivial extension of its maximal degree line subbundle $L$ of degree $0$. Then a Jordan-H\"older  filtration for $E$ is 
$L\subset L\oplus 1 \subset E$ and $Gr(E) = L\oplus 1\oplus V/L$.

\medskip\noindent
(ii) In this case write the holomorphic structure and Higgs field for $(E,t\Phi)$ in block decomposition with respect
to $E=V\oplus 1$:
\[
\bar\partial_E = \begin{pmatrix}\bar\partial_V&0\\0&\bar\partial\end{pmatrix}, \quad
\Phi = \begin{pmatrix} 0&t\Phi_1\\ t\Phi_2 & 0\end{pmatrix}.
\]
A simple computation shows that for 
\[
g_t = \begin{pmatrix} tI_V & 0\\ 0& 1\end{pmatrix}
\]
we have 
\[
\lim_{t\to 0} g_t^{-1}\bar\partial_E g_t = \bar\partial_E,\quad \lim_{t\to 0}g_t^{-1} t\Phi g_t = 
 \begin{pmatrix} 0&\Phi_1 \\ 0 & 0\end{pmatrix}.
\]
This represents the limit provided $(E,\Phi_1)$ is stable as a Higgs bundle. 
The $\Phi_1$-invariant subbundles of $E$ are $1\oplus\im\Phi_1$, $V$ and any line subbundle of $V$. Since $(E,\Phi)$ is
stable as a Higgs bundle and $\deg(V)<0$ the first two already satisfy the slope inequality, as does $\im\Phi_1$. So the additional
condition is that every other line subbundle $L$ of $V$ satisfies $\deg(L)<\tfrac13\deg(V)$. Note that the case $\tau=0$
is covered by part (i).

\medskip\noindent
(iii) First we note that $(V/L\oplus 1,\varphi)$ is a stable Higgs bundle, since $V/L$ is the only $\varphi$-invariant proper subbundle 
and $\deg(V/L)=\tfrac23\deg(V)<0$ hence $\deg(V/L)<\tfrac12\deg(V/L)$. Also $\tfrac12\deg(V/L) =\deg(L)$ hence \eqref{eq:caseiii} 
is a polystable Higgs bundle. We also note
that the restriction of $\Phi_2$ to $L$ is a holomorphic section of $L^{-1}\otimes K(-D_2)$ and this must have zeroes
otherwise the extension \eqref{eq:exactV} splits. Let $D>0$ be the divisor of these zeroes, then $L\simeq K(-D_2-D)$ and
since $\det(V)\simeq \caO(D_1-D_2)$ it follows that $V/L\simeq K^{-1}(D_1+D)$.

Now $V$ can be written as an extension of the form
\begin{equation}\label{eq:LV}
0\to L\to V\to V/L\to 0.
\end{equation}
This gives a $C^\infty$-isomorphism $E \simeq V/L\oplus L\oplus 1$.
With respect to such a decomposition we can write
\[
\bar\partial_E =
\begin{pmatrix}
\bar\partial_{V/L} &0&0\\ \beta&\bar\partial_L &0\\ 0&0&\bar\partial\end{pmatrix},\quad
\Phi = \begin{pmatrix}
0&0&\phi_{13}\\ 0&0&\phi_{23}\\ \phi_{31}& \phi_{32}&0\end{pmatrix}.
\]
Using the gauge transformation
\[
g_t = \begin{pmatrix} t&0&0\\0&1&0\\0&0&1\end{pmatrix}.
\]
a straighforward calculation gives
\[
g_t^{-1}\bar\partial_E g_t =
\begin{pmatrix}
\bar\partial_{V/L} &0&0\\ t\beta&\bar\partial_L &0\\ 0&0&\bar\partial\end{pmatrix},\quad
g_t^{-1}(t\Phi)g_t = \begin{pmatrix}
0&0&\phi_{13}\\ 0&0&t\phi_{23}\\ t^2\phi_{31}& t\phi_{32}&0\end{pmatrix}.
\]
Setting $t=0$ gives \eqref{eq:caseiii}. Note that $\varphi=\varphi_{13}$ is the projection of $\Phi_1:K^{-1}(D_1)\to V$ 
onto $V/L$ and has divisor of zeroes $D_1+D$ since it vanishes when either $\Phi_1=0$ or at the support of $\im\Phi_1\cap L$. 
But $\Phi_2\circ\Phi_1=0$ so the support of $D$ is all the points at which $\im\Phi_1$ is not zero but lies in $L$.

\medskip\noindent
(iv) As in the previous case we first show that the proposed limit \eqref{eq:caseiv} is a stable Higgs bundle. Since the
only proper $\Phi^0$ invariant subbundles are $V/L$ and $V/L\oplus 1$, and since $\deg(V)\leq 0$, the stability condition 
is the single inequality $\tfrac12\deg(V/L)<\tfrac13\deg(V)$, i.e., $\deg(L)> \tfrac13\deg(V)$. Just as in part (iii)
we note that $L\simeq K(-D_2-D)$ where $D>0$ is the divisor of zeroes of $\Phi_2$ restricted to $L$. Hence $V/L\simeq
K^{-1}(D_1+D)$. 

Now to show this gives the limit we use the same argument as part (iii) but with  
\[
g_t = \begin{pmatrix} t^2&0&0\\0&1&0\\0&0&t\end{pmatrix}.
\]
In this case
\[
g_t^{-1}\bar\partial_E g_t = 
\begin{pmatrix}
\bar\partial_{V/L} &0&0\\ t^2\beta&\bar\partial_L &0\\ 0&0&\bar\partial\end{pmatrix},\quad
g_t^{-1}(t\Phi)g_t = \begin{pmatrix}
0&0&\phi_{13}\\ 0&0&t^2\phi_{23}\\ t^2\phi_{31}& \phi_{32}&0\end{pmatrix}.
\] 
Setting $t=0$ gives \eqref{eq:caseiv} with $\Phi^0_1=\varphi_{13}$, the projection of $\Phi_1$ onto $(V/L)\otimes K$, 
and $\Phi^0_2=\varphi_{32}$, the restriction of $\Phi_2$ to $L$. The latter has zeroes $D_2+D$ by definition, and the
former has zeroes $D_1+D$ for the same reason as in part (iii).
\end{proof}
Note that this argument works perfectly well even when $D_1,D_2$ already have points in common
(although $D$ will not be the full branch point divisor in that case), and 
it tells us something about singularities of the nilpotent cone. When $C$ consists of length three Hodge bundles
$\caU_C$ is smooth and of dimension $5(g-1)$: this follows either from the Morse index
calculation of Gothen \cite[Prop 3.2]{Got} or the dimension count in \cite{LofM19} for $\oV_c(d_1,d_2)$, given Corollary
\ref{cor:U}. Any point on $C$ which is a limit point as $t\to 0$ (i.e., from some $\caU_{C'}$ for $C'\neq C$) lies on both the stable 
and unstable manifolds of $C$, which are transverse, and therefore the tangent space at this point has dimension greater than $\caU_C$. 
Thus such limit points are singular points of the nilpotent cone.
We can ask whether every length three Hodge bundle of the form \eqref{eq:caseiv} (i.e., with common divisor $D$) is
such a singular point. The answer is no: by the following result there are many such Hodge bundles which are not limit
points.
Recall from \cite{LofM19} that the necessary and sufficient conditions for \eqref{eq:caseiv} to give a stable Higgs bundle are
the inequalities
\begin{equation}\label{eq:ineq}
0\leq 2d_1'+d_2' < 6(g-1),\quad 0\leq d_1'+2d_2'<6(g-1),
\end{equation}
where $d_j'=\deg(D_j+D)$.
\begin{prop}
Let $D_1,D_2,D$ be effective divisors, $D>0$, set $D_j'=D_j+D$ and suppose their degrees $d_1',d_2'$ satisfy 
$d_1'\leq d_2'$ and $d_1'+d_2'<2(g-1)$. 
Then the Hodge bundle with $V$ given by \eqref{eq:caseiv} does not lie at the limit as $t\to 0$ of a $\C^\times$-orbit
in a different unstable manifold.
\end{prop}
\begin{proof}
First, the two inequalities imply that \eqref{eq:caseiv} corresponds to a stable Higgs bundle with $\tau\geq
0$. By Prop.\ \ref{prop:limit0} such a limit can only be obtained from case (iv), where $V$ has unique maximal destabilizing
line subbundle $L\simeq K(-D_2')$ and $V/L\simeq K^{-1}(D_1')$. The existence of this requires
\[
1\leq \dim H^1(K^2(-D_1'-D_2')) = \dim H^0(K^{-1}(D_1'+D_2')),
\]
using Serre duality.
But the right hand dimension is zero whenever $\deg(K^{-1}(D_1'+D_2'))<0$, i.e., when $d_1'+d_2'<2(g-1)$. 
\end{proof}

\section{Connected components of $\caM(\Sigma,\CH^2)$.}
Our aim is to prove Theorem \ref{thm:connected}. First we must recall some more facts from \cite{LofM19}.
From \cite{LofM19} $\caV$ is a disjoint union $\cup_{(d_1,d_2)}\caV(d_1,d_2)$ where the non-negative integers $d_1,d_2$ satisfy the
inequalities \eqref{eq:ineq}. Points of $\caV(d_1,d_2)$ are parametrised by the data $(c,D_1,D_2,\xi)$ where $D_1,D_2$ are
effective divisors of degree $d_1,d_2$ and $\xi$ denotes the extension class for the extension \eqref{eq:exactV}. This
class can be freely chosen in $H^1(\Sigma_c,K^{-2}(D_1+D_2))$ and this parametrisation gives $\caV(d_1,d_2)$ the structure
of a complex analytic family over $c\in\caT_g$. The fibre over $c\in\caT_g$, which we will denote by $\caV_c(d_1,d_2)$,  
is the holomorphic vector bundle over 
\[
\{(D_1,D_2)\in S^{d_1}\Sigma_c\times S^{d_2}\Sigma_c: D_1\cap D_2=\emptyset\}
\]
with fibre $H^1(\Sigma_c,K^{-2}(D_1+D_2))$. 

We need to show first that this parametrisation is smooth with respect to the smooth structure of $\caM(\Sigma,\CH^2)$
given above, which uses its embedding in $\caT_g\times\caR(\pi_1\Sigma,G)$. Since we are always dealing 
with points which are smooth in the latter we can use the nonabelian Hodge correspondence to identify the smooth locus
of $\caR(\pi_1,G)$ with the smooth locus of Higgs bundle moduli space $\caH(\Sigma_c,G)$.  It therefore suffices to prove the following
lemma. 
\begin{lem}\label{lem:Z}
For each fixed $c\in\caT_g$ we get a holomorphic embedding
\[
\caZ:\oV_c(d_1,d_2)\to \caH(\Sigma_c,G),
\]
by assigning to $(D_1,D_2,\xi)$ the Higgs bundle $(E,\Phi)$ for which $E=V\oplus 1$ where $V$ and $\Phi$ arise from the 
extension $\xi$ using \eqref{eq:exactV}.
\end{lem}
We know from \cite{LofM19} that $\caZ$ is one-to-one, so it suffices to show that it is a holomorphic immersion. 

Before we begin the proof we need to fix a convention for which way to represent $1$-cocycles with values in a vector bundle. For a
vector bundle $F$ with model fibre $A$ and local transition relations $\psi_i=a_{ij}\psi_j$ between local trivialisations
$\psi_i$ over a Leray cover $\{U_i\}$, let $(\xi_{ij},U_i,U_j)$ be a $1$-cocycle for a class in $H^1(F)$, i.e., 
$\xi_{ij}\in\Gamma(U_i\cap U_j,F)$ satisfy  $\xi_{ij}+\xi_{jk}=\xi_{ik}$. Then we choose to represent this by the local functions
\[
c_{ij}=\psi_i(\xi_{ij}): U_i\cap U_j\to A,
\]
and the $1$-cocycle conditions are equivalent to
\[
c_{ij} + a_{ij}c_{jk} = c_{ik}.
\]
Note that the opposite convention, to use $\psi_j(\xi_{ij})$ instead, is used in
Gunning \cite{Gun}. In what follows we need this for $H^1(\caA^0)$ and $H^1(K^{-2}(D_1+D_2))$.
\begin{proof}
To describe the map $\caZ$ concretely about a point $(D_1,D_2,\xi)$, let $P_1,\ldots,P_d$ be the points in the support of
$D_1+D_2$ and let $U_0,\ldots,U_d$ be the Leray cover of $\Sigma_c$ for which 
each $U_j$ for $j\geq 1$ is an open disc about $P_j$, with $U_j\cap U_k=\emptyset$ for $j\neq k$, and $U_0\cap U_j$
is an annulus excluding an open disc about $P_j$. Since $V$ is determined by the extension class $\xi$,
$V$ has $1$-cocycle $g=\{(g_{0j},U_0,U_j):1\leq j\leq d\}$ 
given by
\begin{equation}\label{eq:g}
g_{0j} = \begin{pmatrix} \alpha_{0j} & \lambda_{0j}\\ 0 &\beta_{0j}\end{pmatrix}
\end{equation}
where
\[
\alpha_{0j} = z^{-n_j}\frac{dz_0}{dz_j},\quad 
\beta_{0j}= z_j^{m_j}\frac{dz_j}{dz_0},\quad
\lambda_{0j}= \beta_{0j}\xi_{0j}dz_0^2.
\]
Here, for $j\geq 1$, $z_j$ is a local parameter in $U_j$ centred at $P_j$ and $n_j$ is the degree of $D_1$ at $P_j$
while $m_j$ is the degree of $D_2$ at $P_j$. By $dz_0$ we simply mean a non-vanishing holomorphic $1$-form on $U_0$.
By $\xi_{0j}$ we mean the local section of $K^{-2}(D_1+D_2)$ over $U_0\cap U_j$ which comprises a $1$-cocycle representing 
the extension class $\xi$.  

The $1$-cocycle $g$ determines a trivialisation $\chi_j$ of $V$, and dual $\chi^*_j$ of $V^*$, over $U_j$. 
When we think of the Higgs field components as $\Phi_1\in H^0(V\otimes K)$ and $\Phi_2\in H^0(V^*\otimes K)$ these 
satisfy
\[
\chi_0(\Phi_1) = \begin{pmatrix} dz_0 \\ 0\end{pmatrix},\quad
\chi^*_0(\Phi_2) = \begin{pmatrix} 0 & dz_0 \end{pmatrix},
\]
and for $j\geq 1$
\begin{equation}\label{eq:chi(Phi)}
\chi_j(\Phi_1) = \begin{pmatrix} z_j^{n_j} dz_j \\ 0\end{pmatrix},\quad
\chi^*_j(\Phi_2) = \begin{pmatrix} 0 & z_j^{m_j} dz_j \end{pmatrix}.
\end{equation}

An open neighbourhood around $D_1\in S^{d_1}\Sigma_c$ (respectively $D_2\in S^{d_2}\Sigma_c$) is determined 
by collection of monic polynomials $z_j^{n_j}+u_j(z_j)$ (respectively $z_j^{m_j}+v_j(z_j)$) whose zeroes lie in $U_j$. Note that
if $P_j$ is not in the support of $D_1$ then $n_j=0$ and $u_j$ is the zero polynomial (and likewise for $m_j$, $v_j$ when 
$P_j$ is not in the support of $D_2$).  The coefficients of 
the polynomials $u_j(z_j),v_j(z_j)$ provide a local chart about $(D_1,D_2)$. At a pair 
$(D_1',D_2')$ in this neighbourhood of $(D_1,D_2)$ any $\xi'\in
H^1(K^{-2}(D_1'+D_2'))$ can be represented by a $1$-cocycle $\{(\xi'_{0j},U_0,U_j)\}$. The map $\caZ$ maps $(D_1',D_2',\xi')$
to the Higgs bundle constructed as above but with $z_j^{n_j}$ replaced by $z_j^{n_j}+u_j(z_j)$, $z_k^{m_k}$ replaced by
$z_k^{m_k}+v_k(z_k)$, and $\xi_{0j}$ replaced by $\xi'_{0j}$. It is easy to see that smooth variations of these parameters
result in smooth variations of the Higgs bundle. 

Our aim now is to show that $d\caZ$ has trivial kernel at each point. First we describe the tangent space to
$\oV_c(d_1,d_2)$ at a given point. Since $\oV_c(d_1,d_2)$ is a vector bundle over $S^{d_1}\Sigma_c\times S^{d_2}\Sigma_c$ 
we can fix a local trivialisation
over a neighbourhood of $(D_1,D_2,\xi)$ and identify the tangent space at that point with 
\[
T_{D_1} S^{d_1}\Sigma_c \oplus T_{D_2} S^{d_2} \oplus H^1(K^{-2}(D_1+D_2)).
\]
We will write elements of this in the form $(w,y,\eta)$. Further, the vector $w$ (and similarly $y$) can be described as
follows. It is well-known that
\[
T_{D_1} S^{d_1}\Sigma_c = \oplus_j (\fm_{P_j}/\fm_{P_j}^{n_j+1}),
\]
where $\fm_{P_j}\subset\caO_{P_j}$ is the maximal ideal of locally holomorphic functions about $P_j$ which vanish at $P_j$. We will identify
\[
\fm_{P_j}/\fm_{P_j}^{n_j+1} \simeq \{z_jw_j(z_j): w_j\in\C[z_j],\ \deg(w_j)\leq n_j-1\},
\]
and therefore write $w\in T_{D_1} S^{d_1}\Sigma_c$ as a tuple $w=(w_1,\ldots,w_k)$ of polynomials. This is tangent to the curve 
$D_1(t)$ on $S^{d_1}\Sigma_c$ for which $D_1(t)$ is the divisor of zeroes given by the locally defined polynomial $z_j^{n_j} + t w_j(z_j)$ 
in $U_j$ (for $t$ sufficiently close to zero).

Given such a tangent vector $(w,y,\xi)$, let $(D_1(t),D_2(t),\xi(t))$ be the curve it is tangent to, obtained as above using
$u_j=tw_j$, $v_j=ty_j$ and $\xi(t)=\xi +t\eta$ (using the local trivialisation of $\oV_c(d_1,d_2)$). Let
$\caZ(t)=(E(t),\Phi(t))$ denote the image curve, with $E(t)=V(t)\oplus 1$, and let $g_{0j}(t)$ denote the transition functions for $V(t)$,
in the form \eqref{eq:g}. These have entries
\[
\alpha_{0j}(t) = \frac{1}{z^{n_j}+tw_j}\frac{dz_0}{dz_j},\quad
\beta_{0j}(t) = (z^{m_j}+ty_j)\frac{dz_j}{dz_0},\quad
\lambda_{0j}(t) = \beta_{0j}\xi(t)_{0j}dz_0^2.
\]
In particular, $\chi_0(\Phi_1)$, $\chi^*_0(\Phi_2)$ are time-independent.
The derivative $\caZ'(0)$ is represented by a hypercohomology class in
\[
\H^1(\caA^*)\simeq \frac{\{(A,B)\in \caC^1(\caA^0)\oplus \caC^0(\caA^1):\delta A=0,\ \delta B=[\Phi, A]\}}{
\{(\delta C,[\Phi,C]):C\in \caC^0(\caA^0)\}},
\]
following \cite{BisR}. Specifically, it is the class of a pair $(A,B)$ for which $A$ is represented by the $1$-cocyle 
\begin{equation}\label{eq:g-cocycle}
g'(0)g^{-1}=
\begin{pmatrix} \alpha'(0)\alpha^{-1} & -\lambda\alpha'(0)\alpha^{-1}\beta^{-1} +\lambda'(0)\beta^{-1}\\
0&\beta'(0)\beta^{-1}\end{pmatrix}
\end{equation}
where we have dropped the Cech cocycle subscripts for notational simplicity. If we write $B=(B_1,B_2)$ where $B_j=\Phi_j'(0)$ then
explicit calculation gives
\[
\chi_0(B_1)=\begin{pmatrix} 0 \\ 0\end{pmatrix},\quad 
\chi^*_0(B_2) = \begin{pmatrix} 0&0 \end{pmatrix},
\]
\begin{equation}\label{eq:chi(B)}
\chi_j(B_1) = \begin{pmatrix} w_jdz_j\\ 0\end{pmatrix},\quad
\chi^*_j(B_2)=\begin{pmatrix} 0 & y_jdz_j\end{pmatrix}.
\end{equation}
If $\caZ'(0)=0$ there must exist $C\in \caC^0(\caA^0)$ for which $A=\delta C$ and $B=[\Phi,C]$. We may assume $C$ is
represented by an upper triangular $0$-cocycle of the form
\[
\begin{pmatrix} a_j & b_j \\ 0 & c_j\end{pmatrix},
\]
where the entries are holomorphic functions in $U_j$.
The Lie bracket $[\Phi,C]$ corresponds to the pair $(-C\Phi_1,\Phi_2 C)$ and calculation gives
\[
\chi_j(C\Phi_1) = \begin{pmatrix} a_j z^{n_j}dz_j \\ 0 \end{pmatrix},\quad
\chi_j^*(\Phi_2 C) = \begin{pmatrix} 0 & c_j z^{m_j}dz_j\end{pmatrix}.
\]
Comparing this with \eqref{eq:chi(B)}, since $\deg(w_j)<n_j$ and $\deg(y_j)<m_j$ it follows that $B=[\Phi,C]$ if and only if 
both sides are zero, i.e., $a_j=0=c_j$ and $w_j,y_j$ are both identically zero. This in turn means that the $1$-cocycle in
\eqref{eq:g-cocycle}, which represents $A$, is
\[
\begin{pmatrix} 0 & \eta_{0j} dz_0^2\\ 0 & 0\end{pmatrix}.
\]
Now $A=\delta C$ is the condition that the $1$-cocycle $\eta\in H^1(K^{-2}(D_1+D_2))$ is trivial.
Thus the kernel of $d\caZ$ is trivial.
\end{proof}
It follows that each $\caV(d_1,d_2)$ is smooth and connected as a submanifold of $\caV$, and each is open since they have the
same dimension. Thus we conclude:
\begin{lem}
The smooth submanifold $\caV\subset\caM(\Sigma,\CH^2)$ is a disjoint union of connected components $\caV(d_1,d_2)$.
\end{lem}
To complete the proof of Theorem \ref{thm:connected} we need two more lemmas. Both concern $\tau>0$. The first says that 
$\caV(d_1,d_2)$ is disconnected from $\caW_\tau$ when $d_1\neq 0$. The second says $\caV(0,d_2)$ is connected to
$\caW_\tau$.
\begin{lem}
For $\tau>0$, $\oV(d_1,d_2)\cap\caW_\tau=\emptyset$ when $d_1\neq 0$.
\end{lem}
\begin{proof} 
It suffices to prove that this is true for each fixed conformal structure $c$. So fix $c$ and $\tau>0$, and let $C$ be the critical
manifold of minima for the Higgs field energy $\fE = \|\Phi\|_{L^2}^2$ for this value of $\tau$. Let $C_0\subset C$ be the open
subset of Higgs bundles for which $\Phi$ has no zeroes. Let $C'$ be the critical manifold of $\fE$ for which
$\oV_c(d_1,d_2)$ corresponds to $\oU_{C'}$, the closure of the unstable manifold of the downward gradient
flow of $\fE$ (equally, the unstable manifold for the $\Ct$-action). The assertion is that $\caU_{C'}\cap C_0=\emptyset$
whenever $d_1\neq 0$. Suppose $(E,\Phi)\in \oU_{C'}\cap C_0$, then it is the limit of a sequence $(E_k,\Phi_k)$ of Higgs
bundles in $\caU_{C'}$. We may assume without loss of generality that each $\|\Phi_k\|_{L^2}^2$ lies below the penultimate
critical value of $\fE$, and therefore all $(E_k,\Phi_k)$ lie in the stable manifold $\caS_C$. Taking the limit as
$t\to 0$ is a continuous map (since $\pi:\caS_C\to C$ is a vector bundle) and therefore 
\[
(E,\Phi)= \lim_{k\to\infty} (E_k^0,\Phi_k^0),
\]
where $(E_k^0,\Phi_k^0) = \lim_{t\to 0}(E_k,t\Phi_k)$. By Prop.\  \ref{prop:limit0} each $\Phi_k^0$ has zeroes unless
$d_1=0$. Therefore $(E,\Phi)$ lies in the closed subset $C\setminus C_0$.
\end{proof}
\begin{lem}
For every choice of $0<d_2<3(g-1)$ there is a Higgs bundle $(E,\Phi)$ in $\caV(0,d_2)$ for which
$\lim_{t\to 0}(E,t\Phi)$ lies in $\caW_\tau$ for $\tau=\tfrac23 d_2$.
\end{lem}
\begin{proof}
Since $d_1=0$ the Higgs bundle $(E,\Phi)$ must correspond to an extension of the form
\[
0\to K^{-1}\stackrel{\Phi_1}{\to} V\stackrel{\Phi_2}{\to} K(-D_2)\to 0,
\]
where $\deg(D_2)=d_2$. By Prop.\ \ref{prop:limit0}(ii) it suffices to show we can find such an extension for which $V$ is
a stable bundle. This is equivalent to the existence of an extension
\[
0\to 1\to \lambda \to K^2(-D_2)\to 0,
\]
for which $s(\lambda) = \deg(\lambda)-2\max\deg(L)\geq 1$, where the maximum is taken over all line subbundles of
$\lambda$. Note that such extensions are parametrised by the space $H^1(K^{-2}(D_2))\simeq H^0(K^3(-D_2))^*$.
Such a situation is covered by \cite[Prop 1.1]{LanN}. Specifically, define $d=\deg(K^2(-D_2)) = 4(g-1)-d_2$
and let $s=1$ when $d$ is odd and $s=2$ when $d$ is even. This satisfies the conditions $4-d\leq s\leq d$ required to
apply \cite[Prop 1.1]{LanN} (in particular note that the critical case $4-d=d$ can only occur for $d=2$, in which case $4-d=s=d$). 
Then by \cite[Prop 1.1]{LanN}
$s(\lambda)\geq s$ if and only if the secant variety 
\[
\Sec_{\tfrac12(d+s-2)}(\Sigma_c)\subseteq \P H^0(K^3(-D_2))^*,
\]
is a proper subvariety. Since $\deg(K^3(-D_2)) > 2(g-1)$ the Riemann-Roch theorem gives the dimension of $\P H^0(K^3(-D_2))^*$
as $5(g-1)-d_2-1=d+g-2$. By \cite{Lan} this secant variety has dimension $d+s-3\leq d-1$ is therefore a proper subvariety for
$g\geq 2$.
\end{proof}

\end{document}